\documentclass[10pt,reqno]{amsart}
\usepackage{amsmath,amsopn,amssymb,amsthm,multicol}
\usepackage{hyperref}
\usepackage{graphicx}

\DeclareMathOperator{\Ad}{Ad}

\DeclareMathOperator{\Exp}{exp}

\DeclareMathOperator{\C}{constant}

\newcommand{\fr}{\mathfrak}
\newcommand{\R}{\mathbb{R}}
\renewcommand{\C}{\mathbb{C}}
\newcommand{\al}{\alpha}

\DeclareMathOperator{\SO}{SO}
\DeclareMathOperator{\Sp}{Sp}

\DeclareMathOperator{\U}{U}

 \newtheorem{lemma} {Lemma} [section]
\newtheorem{theorem}[lemma]{Theorem} 
\newtheorem{remark}[lemma] {Remark} 
\newtheorem{prop} [lemma]{Proposition}  
\newtheorem{definition}[lemma] {Definition} 
\newtheorem{corol}[lemma] {Corollary} 
\newtheorem{example}[lemma] {Example}

\begin{document}

\title{Geodesics in generalized Wallach spaces} 
\author{Andreas Arvanitoyeorgos  and Nikolaos Panagiotis Souris}
\address{University of Patras, Department of Mathematics, GR-26500 Patras, Greece}
\email{arvanito@math.upatras.gr}
\email{ntgeva@hotmail.com}
\medskip

   \begin{abstract}
We study geodesics in 
generalized Wallach spaces which are expressed as orbits of products of three exponential terms.  These are homogeneous spaces $M=G/K$ whose isotropy representation decomposes into a direct sum of three submodules $\frak{m}=\frak{m}_1\oplus\frak{m}_2\oplus\frak{m}_3$,
  satisfying the relations $[\frak{m}_i,\frak{m}_i]\subset \frak{k}$.
Assuming that the submodules $\frak{m}_i$ are pairwise non isomorphic, we study geodesics on such spaces of the form $\gamma (t)=\exp (tX)\exp (tY)\exp (tZ)\cdot o$, where $X\in\fr{m}_1, Y\in\fr{m}_2, Z\in\fr{m}_3$ ($o=eK$), with respect to a $G$-invariant metric.
Our investigation imposes certain restrictions on the $G$-invariant metric, so the geodesics turn out to be orbits of two exponential terms.
 We  give a point of view using Riemannian submersions.
 As an application, we describe geodesics in generalized flag manifolds with three isotropy summands and with 
second Betti number $b_2(M)=2$, and in the Stiefel manifolds $SO(n+2)/S(n)$.  We relate our results to geodesic orbit spaces (g.o. spaces).


 \end{abstract}  

\maketitle


 \section{Introduction}
\markboth{Andreas Arvanitoyeorgos and Nikolaos Panagiotis Souris}{Geodesics in homogeneous spaces with three isotropy summands}

Geodesic curves  in a Riemannian manifold are studied not only for their geometrical implications but  for their physical  significance as well (\cite{Arn}).
For a Riemannian homogeneous manifold $(M=G/K, g)$ where $g$ is a $G$-invariant metric, one naturally studies homogeneous geodesics through the origin $o=eK$, i.e. orbits of a one-parameter subgroup of the Lie group $G$.
These are curves of the form
\begin{equation}\label{orbit}
\gamma (t)=\Exp (tX)\cdot o,
\end{equation}
 where $X$ is a non zero vector in the Lie algebra $\fr g$ of $G$.
 Curves of this form had been studied long time ago by many mathematicians such as R. Herman, B. Kostant, E.B. Vinberg.
 
 Riemannian homogeneous spaces of special significance are those whose geodesics are of the form
 (\ref{orbit}), and these are known as geodesic orbit spaces (g.o. spaces).
 The study of these spaces was initiated by O. Kowalski and L. Vanhecke in \cite{KV} and since then they have been investigated in depth by several authors both for Riemannian and semi-Riemannian manifolds (e.g.  \cite{DKN}, \cite{Ta}, \cite{Du}).  
 In particular, for $G$ a simple Lie group, in \cite{AA} the first author and D. Alekseevsky 
 classified generalized flag manifolds with homogeneous geodesics (i.e. which  are g.o. spaces).  The list turned out to be quite short, namely the manifold $\SO (2\ell +1)/\U (\ell)$ of complex structures in $\R ^{2\ell +2}$ and the complex projective space $\C P^{2\ell -2}=\Sp (\ell)/\U (1)\times \Sp (\ell -1)$.  Notice that the isotropy representation of these spaces has two isotropy summands and their associated painted Dynkin diagram has one black root of Dynkin mark 2. 
 
Since a general description of  geodesics in a homogeneous space remains a difficult but an important problem,
it is natural to search for geodesics  that generalize the concept of homogeneous geodesics.
In \cite{Do} R. Dohira gave a description of geodesics in homogeneous spaces
$M=G/K$ whose isotropy representation decomposes into two irreducible summands
$\fr{m}\cong T_{o}M=\fr{m}_1\oplus\fr{m}_2$, with respect to the  (unique up to scalar) diagonal metric.  
His result is the following:

\begin{theorem}\label{Dohira} {\rm (\cite{Do})} Let $M=G/K$ be a homogeneous space with two isotropy summands $\fr{m}=\frak{m}_1\oplus \frak{m}_2$ equipped with the diagonal metric $(1,c)$, where $c>0$.  Assume that the following relations are satisfied:
\begin{equation*} 
\begin{array}{lll} 
[\frak{k},\frak{m}_i]\subset \frak{m}_i \qquad &&[\frak{m}_1,\frak{m}_1]\subset \frak{k} \oplus \frak{m}_2\ \\\\

[\frak{m}_1,\frak{m}_2]\subset \frak{m}_1 \qquad &&[\frak{m}_2,\frak{m}_2]\subset \frak{k}\\
\end{array}
\end{equation*}
  Then the unique geodesic passing through $o$ with $\dot{\gamma}(0)=X_1+X_2$, $X_i\in \frak{m}_i$ is given by
\begin{equation}\label{Doh}
\gamma(t)=\exp t(X_1+cX_2)\exp (1-c)tX_2\cdot o.
\end{equation}
\end{theorem}
Every generalized flag manifold with two isotropy summands satisfies the condition of Theorem \ref{Dohira}.  These spaces have been classified in terms of painted Dynkin diagrams in \cite{AC}, therefore using Dohira's result it is possible to describe all geodesics of the form (\ref{Doh}) in these spaces.
Observe that the geodesics
(\ref{Doh}) are of the form
\begin{equation}\label{2exp}
\gamma (t)=\exp (tX)\exp (tY) \cdot o,\quad X\in\fr{m}_1, Y\in\fr{m}_2.
\end{equation}

Therefore, it is natural to try to extend such a  result to homogeneous spaces $M=G/K$ whose isotropy representation decomposes into three (or more) irreducible summands, i.e.
$$
\fr{m}\cong T_{o}M=\fr{m}_1\oplus\fr{m}_2\oplus\fr{m}_3.
$$
To simplify our study we will assume  that $G$-invariant metrics $g$ on these spaces are
determined by
 $\Ad(K)$-invariant inner products on $\fr{m}$ of the form
 \begin{equation}\label{metric}
 \langle \ ,\ \rangle=\lambda _1\left.(-B)\right|_{\fr{m}_1} +\lambda _2\left.(-B)\right|_{\fr{m}_2}
 +\lambda _3\left.(-B)\right|_{\fr{m}_3}, \quad \lambda_i>0,
\end{equation}
 where $B$ is  the Killing form of $\fr{g}$.  
 For large classes of homogeneous spaces this is a reasonable assumption.
 We denote these metrics by $(\lambda_1,\lambda_2,\lambda_3)$.  
 
 In the present paper we  study geodesics of the form
\begin{equation}\label{3exp}
\gamma (t)=\exp (tX)\exp (tY)\exp (tZ)\cdot o,\quad X\in\fr{m}_1, Y\in\fr{m}_2, Z\in\fr{m}_3
\end{equation}
with respect to the metrics (\ref{metric}) for  the {\it generalized Wallach spaces}.  
These are homogeneous spaces
$M=G/K$ whose isotropy representation $\frak{m}$ decomposes into three $\operatorname{Ad}(K)$-invariant irreducible submodules 
$\frak{m}=\frak{m}_1\oplus\frak{m}_2\oplus\frak{m}_3$ with the property
$[\frak{m}_i, \frak{m}_i]\subset\frak{k}$.
They were introduced by
Yu. G. Nikonorov, E. D. Rodionov and V. V. Slavskii in \cite{NRS} and they were earlier referred to in the literature as {\it three-locally-symmetric spaces} (e.g. \cite{LNF}), because they generalize the property $[\frak{m}, \frak{m}]\subset\frak{k}$ for a locally symmetric homogeneous space $G/K$.
Notice that some of the submodules $\frak{m}_i$ could be equivalent in general, but here {\it we assume that they are pairwise non equivalent}.
Despite their simple description a complete  classification of these spaces was given only very recently by Yu.G. Nikonorov in \cite{Ni2} and independently by Z. Chen, Y. Kang, K. Liang in \cite{Ch-Ka-Li}.

Our analysis leads us to the conclusion that  the form  (\ref{3exp}) of the geodesics we are looking for imposed  restrictions to the $G$-invariant metrics (\ref{metric}).
We prove the following:

 \begin{theorem}\label{maintheorem} Let $M=G/K$ be a generalized Wallach space.
If the $G$-invariant metric  (\ref{metric}) has one of the forms 
 $(1,1,c)$, $(1,c,1)$ or $(c,1,1)$ $(c>0)$, then the unique geodesic $\gamma (t)$ passing through $o$ with 
$\dot{\gamma}(0)=X_1+X_2+X_3 \in T_o(G/K)$, $X_i\in \frak{m}_i$, is given by

\begin{eqnarray}\label{main}
\gamma(t)&=&\exp t(X_1+X_2+cX_3)\exp (t(1-c)X_3) \cdot o,\ \mbox{or}\\
\gamma(t)&=&\exp t(X_1+cX_2+X_3)\exp (t(1-c)X_2) \cdot o, \ \mbox{or}\\
\gamma(t)&=&\exp t(cX_1+X_2+X_3)\exp (t(1-c)X_1) \cdot o
\end{eqnarray}
respectively.  
\end{theorem}
As seen in this theorem, even though we look for geodesics which are orbits of products of three exponential factors, we finally obtain  geodesics which are orbits of   products of two exponential factors.  We give an explanation of this in the last section of the paper.

The above description of geodesics can be applied to various classes of homogeneous spaces, for example for certain
 generalized flag manifolds  with three isotropy summands.  These spaces have been classified   by M. Kimura in \cite{Ki} and  can be split into two classes, depending on whether the
 second Betti number of $M$ is $b_2(M)=1$ or $b_2(M)=2$. Equivalently, their painted Dynkin diagram contains one or two black roots respectively.  They are those with   $b_2(M)=2$ that
satisfy the conditions of Theorem \ref{maintheorem}, which in this case implies the following:
  
\begin{corol} Let $M=G/K$ be a generalized flag manifold with three isotropy summands and with second Betti number $b_2(M)=2$.  Then the geodesics are described by Theorem \ref{maintheorem}. 
\end{corol}

Another class of homogeneous spaces for which Theorem \ref{maintheorem} can be applied, are the real Stiefel manifolds $V_2\mathbb{R}^{n+2}=SO(n+2)/SO(n)$ of orthonormal $2$-frames in $\mathbb{R}^{n+2}$.
Their isotropy representation decomposes into  three irreducible summands, two of which are equivalent.
However, it was proved by M. Kerr in \cite{Ke} that any $SO(n+2)$-invariant metric can be represented by an inner product of the form (\ref{metric}), hence we obtain the following:

\begin{corol}  The geodesics in the Stiefel manifolds $SO(n+2)/SO(n)$ are described by Theorem \ref{maintheorem}.
\end{corol}

The paper is organized as follows:
In Section 2 we obtain results about geodesics in any Riemannian homogeneous space $(M=G/K, g)$ which have their own interest.  More precisely,
let $\gamma$ be a curve in $M=G/K$ which is the projection of a curve $\alpha$ on $G$,  let $W\in\frak{m}$ and let $\hat{W}$ be a certain vector field
on $M$ canonically associated to $W$ (cf. Definition \ref{v}) .
We introduce the real function $G_W(t)=g(\hat{W}_{\gamma(t)},\nabla_{\dot{\gamma}(t)}\dot{\gamma}(t))_{\gamma(t)}$ and we prove that $\gamma$ is a geodesic on $M=G/K$ if and only if $G_W(t)=0$ for every vector $W\in \fr{m}$ and every $t\in \mathbb R$ (cf. Proposition \ref{nikos}).
In Section 3 we consider geodesics of the form (\ref{3exp}) and we express the function
$G_W(t)$ in a more convenient form for our calculations (cf. Proposition \ref{generalform}).

In Section 3 we define the generalized Wallach spaces and in Section 4 we
 prove Theorem \ref{maintheorem}.
We proceed as follows:
Let $(M=G/K,g)$ be a generalized Wallach space, where $g$ is a $G$-invariant metric of the form $(1, \lambda_2, \lambda_3)$ (we observe  that for our purposes this is not a restriction).  Let $\gamma$ be the unique geodesic on $M$ of the form (\ref{3exp}) with $\gamma(0)=o\in M$ and $\dot{\gamma}(0)=X_1+X_2+X_3$, $X_i\in \frak{m}_i$.  Then we conclude that
\begin{equation*} 
G_W(0)=0\ \ \mbox{and}\ \  \dot{G}_W(0)=0.
 \end{equation*}
We elaborate on the above relations and we obtain a system of equations for the parameters $\lambda_2, \lambda_3$.  By solving this system we write the vectors $X,Y,Z$ in (\ref{3exp}) as linear expressions  of $X_1,X_2,X_3$, and obtain some restrictions on the parameters $\lambda_i$.  Therefore, by assuming that a curve of the form (\ref{3exp}) is a geodesic, where $X,Y,Z$ are expressed as linear combinations of $X_1,X_2,X_3$, we obtain all possible forms of $G$-invariant metrics and their corresponding geodesics, as stated  in  Theorem \ref{maintheorem}.  Then it is straightforward  to show that each curve is indeed a geodesic on $(M=G/K,g)$.
Finally, in Section 5 we interpret our results by using Riemannian submersions of the generalized Wallach spaces over locally symmetric spaces.

\medskip
{\bf Acknowledgements.} 
The first author had useful discussion with Professor Yu.G. Nikonorov.
The work was supported by Grant $\# E.037$ from the Research Committee of the University of Patras
(Programme K. Karatheodori).  It was completed while the first author was a visiting scholar at Tufts University during Fall 2014.  The authors thank the referee for a suggestion which led to the formulation of Proposition \ref{remark}.

\section{Geodesics in Homogeneous Spaces}

\subsection{A description of geodesics}

Let $M=G/K$ be a homogeneous space where $G$ is a Lie group and $K$ is the isotropy subgroup of a point $o \in M$.  Let $\frak{g}$ and $\frak{k}$ denote the Lie algebras of $G$ and $K$ respectivelly.  We assume that $M$ is a reductive homogeneous space, that is there exists a decomposition $\frak{g}=\frak{k}\oplus\frak{m}$ where $\frak{m}$ is the orthogonal complement of $\frak{k}$ with respect to an inner product on $\frak{g}$ and $[\frak{k},\frak{m}]\subset \frak{m}$.\\
Let $\pi:G\rightarrow G/K$ be the natural projection which is a submersion.  For $g\in G$ we denote by $g \cdot o$ the action of $g$ on the point $o \in M$, that is $g \cdot o=\pi(g)$.  We identify $\frak{m}$ with the tangent space $T_o(G/K)$ of $M$ at $o$, by identifying each $X\in \frak{m}$ with $d\pi_e(X)\in T_o(G/K)$.  Let $X\in \frak{g}$.  For convenience we will use the notation
\begin{equation} X_{\frak{m}}=d\pi_e(X).\end{equation}
Recall the diffeomorphism $c_g:G\rightarrow G$ with $c_g(h)=ghg^{-1}$.  The adjoint representation of $G$ in $\frak{g}$ is given by $\operatorname{Ad}:G\rightarrow \operatorname{Aut}(\frak{g})$ with $\operatorname{Ad}(g)X=(dc_g)_e(X)$, while the adjoint representation of $\frak{g}$ is given by $\operatorname{ad}:\frak{g}\rightarrow \operatorname{End}(\frak{g})$ with $\operatorname{ad}(X)Y=(d\operatorname{Ad})_e(X)Y$.  The following well known relation will be used throughout this paper:
\begin{equation*}\left.\frac{d}{dt}\right|_{t=0}\operatorname{Ad}(\exp(tX))Y=\operatorname{ad}(X)Y=[X,Y],\quad X,Y\in \frak{g}.\end{equation*}    
 
\noindent  The {\it isotropy representation} of $G/K$, $\operatorname{Ad}^{G/K}:K\rightarrow \operatorname{Aut}(\frak{m})$, is given by 
$\operatorname{Ad}^{G/K}(k)X=(d\tau_k)_oX$, where $k\in K$, $X\in \frak{m}$ and $\tau_k:G/K\rightarrow G/K$ is the left translation in $G/K$, that is $\tau_k(\pi(p))=\pi(kp)$, $p\in G$.
A $G$-invariant metric $g$ on $G/K$ is such that $\tau_k$ is an isometry on $G/K$.  
In this case $(G/K, g)$ is called a {\it Riemannian homogeneous space}.  There is a bijection between $G$-invariant metrics $g$ on $G/K$ and $\operatorname{Ad}(K)$-invariant inner products $\langle \ , \rangle$ on $\frak{m}$.
Let $g\in G$ and $X\in \frak{g}$.  We define the vector fields $X^L$ and $X^R$ by

\begin{eqnarray}\label{leftinv} X^L_g=(dL_g)_e(X) \qquad \mbox{and} \qquad
\label{rightinv} X^R_g=(dR_g)_e(X),
\end{eqnarray}
which are left and right invariant respectivelly.\\
The following result is stated in \cite{Do} without proof, so we take the chance to give a proof here.
\begin{lemma} (\cite{Do})  Let $X,Y\in \frak{g}$ and $g\in G$.  Then the following relations are valid:
\noindent \begin{eqnarray}\label{doh1}&(d\tau_{g^{-1}})_{g\cdot o}(d\pi_g)(X^L_g) &=X_{\frak{m}}\\ 
\label{doh2}&(d\tau_{g^{-1}})_{g\cdot o}(d\pi_g)(X^R_g) &=(\operatorname{Ad}(g^{-1})X)_{\frak{m}}\\
\label{doh3}&[X^L, Y^L]_g &=[X,Y]^L_g\\ 
\label{doh4}&[X^L,Y^R]_g &=0. 
\end{eqnarray}
\end{lemma}

\begin{proof}  We will  show relation (\ref{doh2}).  It is
\begin{equation*}
\begin{array}{lll}
&& \displaystyle{(d\tau_{g^{-1}})_{g\cdot o}(d\pi_g)(X^R_g)} =\displaystyle{\left.\frac{d}{dt}\right|_{t=0}(\tau_{g^{-1}}\circ\pi\circ R_g)(\exp(tX))}\ \\\\

&&\ \ \ =\displaystyle{\left.\frac{d}{dt}\right|_{t=0}(g^{-1}\exp(tX)g)\cdot o}=\displaystyle{\left.\frac{d}{dt}\right|_{t=0}(c_{g^{-1}}(\exp(tX)))\cdot o}\ \\\\

&&\ \ \ = \displaystyle{\left.\frac{d}{dt}\right|_{t=0}(\exp((dc_{g^{-1}})_e(tX)))\cdot o}=\displaystyle{\left.\frac{d}{dt}\right|_{t=0}(\exp(t(dc_{g^{-1}})_eX))\cdot o}\ \\\\

&&\ \ \ =\displaystyle{\left.\frac{d}{dt}\right|_{t=0}(\exp(t\operatorname{Ad}(g^{-1})X))\cdot o}=\displaystyle{\left.\frac{d}{dt}\right|_{t=0}\pi(\exp(t\operatorname{Ad}(g^{-1})X))}\ \\\\

&&\ \ \ =(d\pi_e)(\operatorname{Ad}(g^{-1})X)=(\operatorname{Ad}(g^{-1})X)_{\frak{m}}.
\end{array}
\end{equation*}

To show relation (\ref{doh4}), let $f:M\rightarrow \mathbb R$ be a smooth function.  Then $[X^L,Y^R]_gf=X^L_g(Y^Rf)-Y^R_g(X^Lf)$.  We have
\begin{equation*}
\begin{array}{lll}
&& \displaystyle{X^L_g(Y^Rf)} =(dL_g)_e(X)(Y^Rf)=X((Y^Rf)\circ L_g)\ \\\\

&&\ \ \ =\displaystyle{\left.\frac{d}{ds}\right|_{s=0}((Y^Rf)\circ L_g)(\exp(sX))}=\displaystyle{\left.\frac{d}{ds}\right|_{s=0}(Y^Rf)_{g\exp(sX)}}\ \\\\

&&\ \ \ =\displaystyle{\left.\frac{d}{ds}\right|_{s=0}Y^R_{g\exp(sX)}f}=\displaystyle{\left.\frac{d}{ds}\right|_{s=0}(dR_{g\exp(sX)})_e(Y)f}\ \\\\

&&\ \ \ =\displaystyle{\left.\frac{d}{ds}\right|_{s=0}Y(f\circ R_{g\exp(sX)})}=\displaystyle{\left.\frac{d}{ds}\right|_{s=0}\left.\frac{d}{dt}\right|_{t=0}(f\circ R_{g\exp(sX)})(\exp(tY))}\ \\\\

&&\ \ \ =\displaystyle{\left.\frac{d}{ds}\right|_{s=0}\left.\frac{d}{dt}\right|_{t=0}f(\exp(tY)g\exp(sX))}.
\end{array}
\end{equation*}
    
The term $Y^R_g(X^Lf)$ is also equal to $\displaystyle{\left.\frac{d}{ds}\right|_{s=0}\left.\frac{d}{dt}\right|_{t=0}f(\exp(tY)g\exp(sX))}$, thus relation (\ref{doh4}) follows.\\
The proof for relation (\ref{doh1}) is similar to that of (\ref{doh2}).  Relation (\ref{doh3}) holds by the $L_g$-equivalence of left invariant vector fields.
\end{proof} 
 
Let $\alpha:I\subset \mathbb R\rightarrow G$ be a smooth curve in $G$.  Then $\gamma=\pi \circ \alpha:I\rightarrow G/K$ is a smooth curve in $M=G/K$ and $\dot{\gamma}$ is a vector field along $\gamma$.  We extend $\dot{\gamma}$ to a vector field locally in $M$ as follows.
First note that the vector field $\dot{\alpha}$ along $\alpha$ assigns to each point $\alpha(t_0)$ the tangent vector $\dot{\alpha}_{\alpha(t_0)}=\dot{\alpha}(t_0)$.  
Then, since $d\pi_e |_{\frak{m}}:\frak{m}\rightarrow T_o(G/K)$ is a vector space isomorphism, we have that
$\pi:\exp \frak{m}\rightarrow G/K$ is a local diffeomorphism.  Thus, there exists an open neighborhood $U$ of $e$ in $\exp \frak{m}$ such that $\pi |_U:U\rightarrow \pi(U)$ is a diffeomorphism.

 Since $L_{\alpha(t)}$ is a homeomorphism and $U$ is an open neighborhood of $e$, the set $U_{\alpha(t)}=\left\{\alpha(t)g: g\in U\right\}=L_{\alpha(t)}(U)$ is an open neighborhood of $\alpha(t)$ in $\alpha(t)\exp \frak{m}$.  Also, $\pi(U_{\alpha(t)})$ is an open neighborhood of $\gamma(t)$ in $\pi(\alpha(t)\exp \frak{m})=\alpha(t)\exp \frak{m}\cdot o$, which  in turn is open in $\tau_{\alpha(t)}(G/K)=G/K$.  Finally, since $L_{\alpha(t)}$ is a diffeomorphism, we have that $\pi|_{U_{\alpha(t)}}:U_{\alpha(t)}=L_{\alpha(t)}(U)\rightarrow \pi(U_{\alpha(t)})$ is a diffeomorphism.
The {\it extension of $\dot{\gamma}$ in $\pi(U_{\alpha(t)})$}, which we also denote by $\dot{\gamma}$, assigns to each point $\pi(\alpha(t)g)\in \pi(U_{\alpha(t)})$ the tangent vector
 \begin{equation}
 \label{vf}\dot{\gamma}_{\pi(\alpha(t)g)}=d\pi_{\alpha(t)g}(\dot{\alpha}_{\alpha(t)g}),\end{equation}
where we also abuse the notation to denote by $\dot{\alpha}$ the extension of the tangent vector field along $\alpha$.  The above extension is well defined since $\pi|_{U_{\alpha(t)}}$ is 1-1.
The following objects will be of central interest in this paper.

\begin{definition}\label{v}
For any $W\in \frak{m}$ we define the vector field $\hat{W}$ on $\pi(U_{\alpha(t)})\subset G/K$ by 
\begin{equation}\label{avf} 
\hat{W}_{\pi(\alpha(t)g)}=d\pi_{\alpha(t)g}(W^L_{\alpha(t)g}).
\end{equation}
\end{definition}
\begin{definition}
For any $W \in \frak{m}$, we define the function $G_W:\mathbb R\rightarrow \mathbb R$ by\\
\begin{equation} \label{function} G_W(t)=g(\hat{W}_{\gamma(t)},\nabla_{\dot{\gamma}(t)}\dot{\gamma}(t))_{\gamma(t)}.
\end{equation}
\end{definition}
The function $G_W(t)$ can be used to characterize geodesics in a homogeneous space as shown in the next proposition.
\begin{prop} \label{nikos} Let $\alpha:I\rightarrow G$ be a curve in the Lie group $G$.  Then the curve $\gamma=\pi \circ \alpha:I\rightarrow G/K$ is a geodesic in the Riemannian homogeneous space $(G/K, g)$ if and only if $G_W(t)=0$ for all $t\in \mathbb R$ and for all $W\in \frak{m}$.
\end{prop}
\begin{proof}
\noindent By using Koszul's formula we obtain that 
\begin{equation} \label{koz} 
g(V,\nabla_XX)=Xg(V,X)+g(X,[V,X])-\frac{1}{2}Vg(X,X), 
\end{equation}
where $V,X$ are arbitrary vector fields in $M$.\\
Let $X=\dot{\gamma}$ as defined in (\ref{vf}).  Then relation (\ref{koz}) yields
\begin{equation} \label{koz1} g(V,\nabla_{\dot{\gamma}}\dot{\gamma})=\dot{\gamma}g(V,\dot{\gamma})+g(\dot{\gamma},[V,\dot{\gamma}])-\frac{1}{2}Vg(\dot{\gamma},\dot{\gamma}).\end{equation}

By evaluating the function $g(V,\nabla_{\dot{\gamma}}\dot{\gamma})$ at $\gamma(t)\in M$ we obtain
\begin{equation}\label{kozev}
g(V_{\gamma(t)},\nabla_{\dot{\gamma}(t)}\dot{\gamma}(t))_{\gamma(t)}=\dot{\gamma}(t)g(V,\dot{\gamma})+g(\dot{\gamma}(t),[V,\dot{\gamma}]_{\gamma(t)})_{\gamma(t)}-\frac{1}{2}V_{\gamma(t)}g(\dot{\gamma},\dot{\gamma}).
\end{equation}

It follows that $\gamma$ is a geodesic if and only if $\nabla_{\dot{\gamma}(t)}\dot{\gamma}(t)=0$ for every $t\in \mathbb R$, or equivalently if $g(V_{\gamma(t)},\nabla_{\dot{\gamma}(t)}\dot{\gamma}(t))_{\gamma(t)}=0$ for every $t\in \mathbb R$ and for every vector field $V\in M=G/K$.  However, since the metric $g$ is $G$-invariant, our calculations will be restricted on $\frak{m}$.  Thus, without any loss of generality we can choose the arbitrary $V$ to be as in (\ref{avf}).  Therefore, $\gamma$ is a geodesic if and only if $G_W(t)=g(\hat{W}_{\gamma(t)},\nabla_{\dot{\gamma}(t)}\dot{\gamma}(t))_{\gamma(t)}=0.$ 
\end{proof}
\subsection{Geodesics in $G/K$ as orbits of exponential factors.}

We are interested to describe geodesics in $G/K$ of the form $\gamma=\pi \circ \alpha$, where $\alpha:I\rightarrow G$ with $\alpha(t)=\exp(tX)\exp(tY)\exp(tZ)$ for $X,Y,Z\in \frak{m}$.  Our aim is to simplify the right hand side of expression (\ref{kozev}) for the function $G_W(t)$.
For $Z,Y\in \frak{m}$ we define the function $T:\mathbb R\rightarrow \operatorname{Aut}(\frak{g})$ by
\begin{equation} \label{T} 
T(t)=\operatorname{Ad}(\exp(-tZ)\exp(-tY)).
\end{equation}
From now on we will write $T=T(t)$.
\begin{lemma} \label{calculations} Let $X,Y,Z\in \frak{m}$, $\alpha(t)=\exp(tX)\exp(tY)\exp(tZ)$ and $T(t)$ as defined in (\ref{T}).  Then the following relations are valid:
\begin{eqnarray}
\label{lab1} \displaystyle{\left.\frac{d}{dt}\right|_{t=0}T(t)X}&=&[X,Y+Z],\\                          
\label{lab2} \operatorname{Ad}(\exp(tX))X&=&X,\\ 
 \label{lab3} \displaystyle{\left.\frac{d}{ds}\right|_{s=0}\operatorname{Ad}(\exp(-t-s)Z)Y}&=&[TY,Z],\\
 \label{lab4} \displaystyle{\left.\frac{d}{ds}\right|_{s=0}\operatorname{Ad}(\alpha(t+s)^{-1})X}&=&[TX,Z]+[TX,TY].
 \end{eqnarray}
\end{lemma}  
\begin{proof}For (\ref{lab1}), let $c(t)=\exp(-tZ)\exp(-tY)$.  Then

\begin{equation*}
\begin{array}{lll}
\displaystyle{\left.\frac{d}{dt}\right|_{t=0}T(t)X}&=&\displaystyle{\left.\frac{d}{dt}\right|_{t=0}\operatorname{Ad}(c(t))X}=(d\operatorname{Ad})_e(\dot{c}(0))X\ \\\\

&=&\operatorname{ad}(\dot{c}(0))X=[-Y-Z,X]=[X,Y+Z].
\end{array}
\end{equation*}
 
For (\ref{lab2}) we have
\begin{equation*}
\begin{array}{lll}
\displaystyle{\operatorname{Ad}(\exp(tX))X} &=&(dc_{\exp(tX)})_eX=\displaystyle{\left.\frac{d}{ds}\right|_{s=0}\exp (tX) \exp(sX) \exp(-tX)}
\ \\\\
&=&
\displaystyle{\left.\frac{d}{ds}\right|_{s=0}\exp(t+s-t)X}=X.
\end{array}
\end{equation*}
For (\ref{lab3}) we have
\begin{equation*}
\begin{array}{lll}
\displaystyle{\left.\frac{d}{ds}\right|_{s=0}\operatorname{Ad}(\exp(-t-s)Z)Y} &=&\displaystyle{\left.\frac{d}{ds}\right|_{s=0}\operatorname{Ad}(\exp(-sZ))\circ \operatorname{Ad}(\exp(-tZ))}\ \\\\
&=&(d\operatorname{Ad})_e(-Z)(\operatorname{Ad}(\exp(-tZ))Y)
\ \\\\
&=&
(d\operatorname{Ad})_e(-Z)(TY)=[-Z,TY]=[TY,Z].\ \\\\
\end{array}
\end{equation*}
where in the third equality we used (\ref{lab2}).  Finally, for (\ref{lab4}) we have
\begin{eqnarray*}
&&\displaystyle{\left.\frac{d}{ds}\right|_{s=0}\operatorname{Ad}(\alpha(t+s)^{-1})X}\\
&& = \displaystyle{\left.\frac{d}{ds}\right|_{s=0}\operatorname{Ad}(\exp (-tZ) \exp (-sZ) \exp (-tY) \exp (-sY) \exp (-tX) \exp (-sX))X}\\
&&=\displaystyle{\left.\frac{d}{ds}\right|_{s=0}\operatorname{Ad}(\exp (-tZ) \exp (-sZ) \exp (-tY) \exp (-sY))}\\
&& \displaystyle{\quad\circ \operatorname{Ad}( \exp (-tX) \exp (-sX))X}\\
&&=\displaystyle{\left.\frac{d}{ds}\right|_{s=0}\operatorname{Ad}(\exp (-tZ) \exp (-sZ) \exp (-tY) \exp (-sY))X}\\
&&=\displaystyle{\left.\frac{d}{ds}\right|_{s=0}\operatorname{Ad}(\exp (-sZ) \exp (-tZ) \exp (-sY) \exp (tZ) \exp (-tZ) \exp (-tY))X}\\
&&=\displaystyle{\left.\frac{d}{ds}\right|_{s=0}\operatorname{Ad}(\exp (-sZ) \exp (-s\operatorname{Ad}(\exp (-tZ))Y))}\\
&&\displaystyle{\quad\circ \operatorname{Ad}(\exp(-tZ)\exp(-tY))X}\\
&&=\displaystyle{\left.\frac{d}{ds}\right|_{s=0}\operatorname{Ad}(\exp (-sZ) \exp (-s\operatorname{Ad}(\exp (-tZ))Y))\circ TX}\\
&&=[TX,Z]+[TX,\operatorname{Ad}(\exp(-tZ))Y]=[TX,Z]+[TX,TY].
\end{eqnarray*} 
\end{proof}

\begin{lemma} \label{curve} 
Let $c:I \rightarrow \frak{g}$ be a curve in $\frak{g}$.  
Then $\displaystyle{\left.\frac{d}{ds}\right |_0c(s)_{\frak{m}}}=\displaystyle{\left(\left.\frac{d}{ds}\right |_0c(s)\right)_{\frak{m}}}$.
\end{lemma}
\begin{proof} It is
\begin{equation*}
\begin{array}{lll}
\displaystyle{\left.\frac{d}{ds}\right |_0c(s)_{\frak{m}}}&=&\displaystyle{\left.\frac{d}{ds}\right |_0d\pi_e(c(s))}=\displaystyle{d(d\pi_e)_{c(0)}\left(\left.\frac{d}{ds}\right |_0c(s)\right)}\ \\\\

&=&\displaystyle{d\pi _e\left(\left.\frac{d}{ds}\right |_0c(s)\right)}=\displaystyle{\left(\left.\frac{d}{ds}\right |_0c(s)\right)_{\frak{m}}}.
\end{array}
\end{equation*} 
\end{proof}

We can now simplify the function $G_W(t)$.
\begin{prop} \label{generalform} Let $\gamma:I\rightarrow G/K$ be the curve $\gamma(t)=\exp(tX)\exp(tY)\exp(tZ)\cdot o$ in  $G/K$, where $X,Y,Z\in \frak{m}$.  Then for every $W\in \frak{m}$, the function $G_W(t)$ defined in (\ref{function}) can be expressed as
\begin{equation} \label{*}
\begin{array}{lll}
G_W(t)&=&g((TX)_{\frak{m}}+(TY)_{\frak{m}}+Z_{\frak{m}},[W, TX+TY+Z]_{\frak{m}})_o\ \\\\

&+&g(W,[TX, TY+Z]_{\frak{m}}+[TY, Z]_{\frak{m}})_o.
\end{array}
\end{equation}
\end{prop}
\begin{proof}
Recall that $\dot{\gamma}(t)=(d\pi)_{\alpha(t)}(\dot{\alpha}(t))$, where $\alpha(t)=\exp(tX) \exp(tY) \exp(tZ)$, hence we need at first to compute $\dot{\alpha}(t)$.  By differentiating the Lie group product we obtain that

\begin{eqnarray}
&& \displaystyle{\dot{\alpha}(t)} =\displaystyle{\left.\frac{d}{ds}\right|_{s=0}\alpha(t+s)}=\displaystyle{\left.\frac{d}{ds}\right|_{s=0}\exp (t+s)X \exp (t+s)Y \exp(t+s)Z}\nonumber\\
&& =
\displaystyle{\left.\frac{d}{ds}\right|_{s=0}\exp (tX) \exp (tY) \exp(t+s)Z}+\displaystyle{\left.\frac{d}{ds}\right|_{s=0}\exp (tX) \exp (t+s)Y \exp(tZ)}\nonumber \\
&& +
\displaystyle{\left.\frac{d}{ds}\right|_{s=0}\exp (t+s)X \exp (tY) \exp(tZ)}.\label{eq1}
\end{eqnarray}

\noindent Since $\exp(t+s)Z=\exp(tZ)\exp(sZ)$, the first term in (\ref{eq1}) is written as
\begin{equation} \label{ter1}
\begin{array}{lll} 
&&\displaystyle{\left.\frac{d}{ds}\right|_{s=0}\exp (tX) \exp (tY) \exp(t+s)Z}\ \\\\

&&=\displaystyle{\left.\frac{d}{ds}\right|_{s=0}\alpha(t)\exp (sZ)} = \displaystyle{\left.\frac{d}{ds}\right|_{s=0}L_{\alpha(t)}(\exp (sZ))}= (dL_{\alpha(t)})_eZ=Z^L_{\alpha(t)}.
\end{array}
\end{equation}
The third term in (\ref{eq1}) is written as
\begin{equation} \label{ter2} 
\begin{array}{lll}
&&\displaystyle{\left.\frac{d}{ds}\right|_{s=0}\exp (t+s)X \exp (tY) \exp(tZ)}\ \\\\

&&=\displaystyle{\left.\frac{d}{ds}\right|_{s=0}\exp(sX)\alpha(t)}=\displaystyle{\left.\frac{d}{ds}\right|_{s=0}R_{\alpha(t)}(\exp (sX))}=(dR_{\alpha(t)})_eX=X^R_{\alpha(t)},
\end{array}
\end{equation}
and finally the second term of (\ref{eq1}) is equal to  
\begin{equation*}
\displaystyle{\left.\frac{d}{ds}\right|_{s=0}\exp (tX) \exp (tY) \exp(tZ)\exp(-tZ)\exp(sY)\exp(tZ)}.
\end{equation*} 
Therefore, we have that 
\begin{equation} \label{ter3} 
\begin{array}{lll}
&& \displaystyle{\left.\frac{d}{ds}\right|_{s=0}\exp (tX) \exp (t+s)Y \exp(tZ)}=\displaystyle{\left.\frac{d}{ds}\right|_{s=0}\alpha(t) \exp (s\operatorname{Ad}(\exp (-tZ))Y)}\ \\\\

&&\ =(\operatorname{Ad}(\exp (-tZ)Y)^L)_{\alpha(t)}. 
\end{array}
\end{equation}
\noindent By adding (\ref{ter1}), (\ref{ter2}) and (\ref{ter3}) we obtain that
\begin{equation}\label{alpha}
\dot{\alpha}(t)=(X^R+Z^L+(\operatorname{Ad}(\exp (-tZ)Y)^L))_{\alpha(t)}.
\end{equation}

 Next, we calculate each of the three terms in the right hand side of (\ref{kozev}) to obtain the desired expression of the function $G_W(t)$.  By using relations (\ref{alpha}), (\ref{vf}) and (\ref{avf}) the first term of (\ref{kozev}) becomes 
\begin{equation*}
\begin{array}{lll}
&& \displaystyle{\dot{\gamma}(t)g(\hat{W},\dot{\gamma})}=\dot{\gamma}_{\gamma(t)}g(\hat{W},\dot{\gamma})=\displaystyle{\left.\frac{d}{ds}\right|_{s=0}g(\hat{W}_{\gamma(t+s)}, \dot{\gamma}_{\gamma(t+s)})_{\gamma(t+s)}}
\ \\\\
&&\ = 
\displaystyle{\left.\frac{d}{ds}\right|_{s=0}g\bigg(d\pi_{\alpha(t+s)}W^L_{\alpha(t+s)}, d\pi_{\alpha(t+s)}(X^R}+Z^L\ \\\\

&&\ \displaystyle{\qquad\qquad\qquad+(\operatorname{Ad}(\exp(-t-s)Z)Y)^L)_{\alpha(t+s)}\bigg)_{\alpha(t+s)\cdot o}}.
\end{array}
\end{equation*}

Since the metric $g$ is $G$-invariant, the above term is equal to
\begin{eqnarray*}
&&\displaystyle{\left.\frac{d}{ds}\right|_{s=0}g\Big((d\tau_{\alpha^{-1}(t+s)})_{\gamma(t+s)}(d\pi_{\alpha(t+s)})(W^L_{\alpha(t+s)})},\\
&&\ \ (d\tau_{\alpha^{-1}(t+s)})_{\gamma(t+s)}(d\pi_{\alpha(t+s)})(X^R+Z^L+(\operatorname{Ad}(\exp(-t-s)Z)Y)^L)\Big)_o.
\end{eqnarray*}

By using relations (\ref{doh1})-(\ref{doh4}) we obtain that
\begin{equation} \label{temp1} 
\begin{array}{lll}
&& \dot{\gamma}(t)g(\hat{W},\dot{\gamma})=\displaystyle{\left.\frac{d}{ds}\right|_{s=0}g\Big(W, (\operatorname{Ad}(\alpha(t+s)^{-1})X)_{\frak{m}}}\ \\\\

&&\ \ +\displaystyle{(\operatorname{Ad}(\exp (-t-s)Z)Y)_{\frak{m}}\Big)_o}.
\end{array}
 \end{equation}

Therefore, by using Lemmas \ref{calculations}, \ref{curve} and equation (\ref{temp1}), the first term of (\ref{kozev}) becomes
\begin{equation} \label{term1} 
g(W, [TX, Z]_{\frak{m}}+[TX, TY]_{\frak{m}}+[TY, Z]_{\frak{m}})_o.
\end{equation} 
Next, we use  (\ref{vf}), (\ref{avf}), relations (\ref{doh1})-(\ref{doh4}) and the $G$-invariance of the metric to write  the second term of  (\ref{kozev}) as
\begin{equation*}
\begin{array}{lll}
&& \displaystyle{g(\dot{\gamma}(t),[\hat{W},\dot{\gamma}]_{\gamma(t)})_{\gamma(t)}}\ \\\\

&&\ \ =g\Big(d\pi_{\alpha(t)}(X^R+Z^L+(\operatorname{Ad}(\exp (-tZ))Y)^L)_{\alpha(t)},\ \\\\

&&\ \ \ [d\pi_{\alpha(t)}(W^L_{\alpha(t)}), d\pi_{\alpha(t)}(X^R+Z^L+(\operatorname{Ad}(\exp (-tZ))Y)^L)_{\alpha(t)}]\Big)_{{\alpha(t)}\cdot o}
\ \\\\

&&\ \ =g\Big((d\tau_{\alpha(t)^{-1}})_{\alpha(t)\cdot o}(d\pi_{\alpha(t)})(X^R+Z^L+(\operatorname{Ad}(\exp (-tZ))Y)^L)_{\alpha(t)},\ \\\\

&&\ \ \ (d\tau_{\alpha(t)^{-1}})_{\alpha(t)\cdot o}(d\pi_{\alpha(t)})[W^L, X^R+Z^L+(\operatorname{Ad}(\exp (-tZ))Y)^L]_{\alpha(t)}\Big)_{o}\ \\\\

&&\ \ =g\Big((\operatorname{Ad}(\alpha(t)^{-1})X)_{\frak{m}}+Z_{\frak{m}}+(\operatorname{Ad}(\exp(-tZ))Y)_{\frak{m}},\ \\\\

&&\ \ \ (d\tau_{\alpha(t)^{-1}})_{\alpha(t)\cdot o}(d\pi_{\alpha(t)})[W, \operatorname{Ad}(\exp(-tZ))Y+Z]^L_{\alpha(t)}\Big)_{o}\ \\\\

&&\ \ =g\Big((\operatorname{Ad}(\exp(-tZ)\exp(-tY))X)_{\frak{m}}+Z_{\frak{m}}+(\operatorname{Ad}(\exp(-tZ))Y)_{\frak{m}},\ \\\\

&&\ \ \ [W,\operatorname{Ad}(\exp(-tZ))Y+Z]_{\frak{m}}\Big)_o.
\end{array}
\end{equation*}
Therefore, the second term of (\ref{kozev}) is equal to 
\begin{equation} \label{term2} 
g((TX)_{\frak{m}}+Z_{\frak{m}}+(TY)_{\frak{m}}, [W, TY+Z]_{\frak{m}})_{o}.
\end{equation} 
For $X\in \frak{m}$ let $\left\|X\right\|^2=g(X,X)$.  Then by using similar calculations as above, the third term of (\ref{kozev}) becomes 

\small
\begin{equation*}
\begin{array}{lll}
&& -\frac{1}{2}\hat{W}_{\gamma(t)}g(\dot{\gamma},\dot{\gamma})=-\frac{1}{2}\hat{W}_{\gamma(t)}\left\|\dot{\gamma}\right\|^2=-\frac{1}{2}d\pi_{\alpha(t)}(W^L_{\alpha(t)})\left\|\dot{\gamma}\right\|^2\ \\\\
&&\ \ =
\displaystyle{-\frac{1}{2}}W^L_{\alpha(t)}(\left\|\dot{\gamma}\right\|^2\circ \pi)=\displaystyle{-\frac{1}{2}\left.\frac{d}{ds}\right|_{s=0}}\left\|\dot{\gamma}_{\pi(\alpha(t)\exp(sW))}\right\|^2_{\alpha(t)\exp(sW)\cdot o}
\ \\\\
&&\ \ = 
\displaystyle{-\frac{1}{2}\left.\frac{d}{ds}\right|_{s=0}}\Big\|(d\pi_{\alpha(t)\exp(sW)})(X^R+Z^L\ \\\\
&&\ \ \ \ \ \ \ \qquad\qquad+
(\operatorname{Ad}(\exp(-tZ))Y)^L)_{\alpha(t)\exp(sW)}\Big\|^2_{\alpha(t)\exp(sW)\cdot o}\ \\\\
&&\ \ =
\displaystyle{-\frac{1}{2}\left.\frac{d}{ds}\right|_{s=0}}\Big\|(d\tau_{\exp(-sW)\alpha(t)^{-1}})_{\alpha(t)\exp(sW)\cdot o}(d\pi_{\alpha(t)\exp(sW)})(X^R+Z^L\ \\\\
&&\ \ \ \ \ \ \ \qquad\qquad+
(TY)^L)_{\alpha(t)\exp(sW)} \Big\|^2_o\ \\\\
&&\ \ =
\displaystyle{-\frac{1}{2}\left.\frac{d}{ds}\right|_{s=0}}\Big\|(\operatorname{Ad}(\exp(-sW)\alpha(t)^{-1})X)_{\frak{m}}+Z_{\frak{m}}+(TY)_{\frak{m}}\Big\|^2_o\ \\\\
&&\ \ =
\displaystyle{-\frac{1}{2}\left.\frac{d}{ds}\right|_{s=0}}\Big\|(\operatorname{Ad}(\exp(-sW)\exp(-tZ)\exp(-tY))X)_{\frak{m}}+Z_{\frak{m}}+(TY)_{\frak{m}}\Big\|^2_o\ \\\\
&&\ \ =
\displaystyle{-\frac{1}{2}\left.\frac{d}{ds}\right|_{s=0}}\Big\|(\operatorname{Ad}(\exp(-sW))TX)_{\frak{m}}+Z_{\frak{m}}+(TY)_{\frak{m}}\Big\|^2_o.
\end{array}
\end{equation*}
By similar computations as in the proof of Lemma \ref{calculations} it follows that
$$
\displaystyle{\left.\frac{d}{ds}\right|_{s=0}\operatorname{Ad}(\exp(-sW))\circ TX=[TX,W]},
$$
and by using Lemma \ref{curve}, the third term of  (\ref{kozev}) becomes
\begin{equation} \label{term3}
 -g([TX, W]_{\frak{m}}, (TX)_{\frak{m}}+Z_{\frak{m}}+(TY)_{\frak{m}})_o.
\end{equation}
We finally add (\ref{term1}), (\ref{term2}), (\ref{term3}) to obtain that

\begin{equation*}
\begin{array}{lll}
G_W(t)&=&g((TX)_{\frak{m}}+(TY)_{\frak{m}}+Z_{\frak{m}},[W, TX+TY+Z]_{\frak{m}})_o\ \\\\

&+&g(W,[TX, TY+Z]_{\frak{m}}+[TY, Z]_{\frak{m}})_o,
\end{array}
\end{equation*}
which proves the proposition.
\end{proof}

\section{Generalised Wallach spaces}
 Let $M=G/K$ be a compact homogeneous space where $G$ is a semisimple Lie group which acts almost effectively on $M$ and $K$.  Recall the reductive decomposition $\frak{g}=\frak{k}\oplus \frak{m}$ where $\frak{g},\frak{k}$ are the Lie algebras of $G,K$ respectively and $\frak{m}\equiv T_oM$.  Let $B$ be the Killing form of $\frak{g}$.  

\begin{definition} (\cite{NRS}) A generalised Wallach space is a homogeneous space $M=G/K$ whose isotropy representation $\frak{m}$ decomposes into three $\operatorname{Ad}(K)$-invariant irreducible and pairwise orthogonal  submodules as
\begin{equation}\label{dec}
\frak{m}=\frak{m}_1\oplus\frak{m}_2\oplus\frak{m}_3,
\end{equation}
 which satisfy the relations
\begin{equation}\label{brackets}
[\frak{m}_i,\frak{m}_i]\subset \frak{k}, \quad i=1,2,3.\
\end{equation}
\end{definition}  
Despite their simple description a complete  classification of these spaces was given only very recently by Yu.G. Nikonorov in \cite{Ni2} and \cite{Ch-Ka-Li}.
We give some examples of generalised Wallach spaces.
\begin{example}
\textnormal{
The Wallach spaces 
\begin{equation*}SU(3)/T_{\mbox{max}}, \quad Sp(3)/(SU(2)\times SU(2)\times SU(2)), \quad F_4/\operatorname{Spin}(8).
\end{equation*}
}
\end{example}
\begin{example}\label{gfm}
\textnormal{The generalized flag manifolds 
\begin{equation*}SU(l+m+n)/S(U(l)\times U(m)\times U(n)), \quad SO(2l)/(U(1)\times U(l-1)),\quad E_6/(U(1)\times U(1)\times \operatorname{Spin}(8)).\end{equation*}}
\end{example}
\begin{example}
\textnormal{The homogeneous spaces
\begin{equation*} SO(l+m+n)/(SO(l)\times SO(m)\times SO(n)),\quad Sp(l+m+n)/(Sp(l)\times Sp(m)\times Sp(n)).
\end{equation*}
}
\end{example}
\begin{example}
\textnormal{The Stiefel manifolds  $SO(n+2)/SO(2)$.}\end{example}

The following property is mentioned in \cite[p. 169]{Ni} without proof.
\begin{lemma} Let $M=G/K$ be a generalized Wallach space.  Then the  submodules $\frak{m}_i$ $(i=1,2,3)$  satisfy the relations
\begin{equation}\label{rel3}[\frak{m}_1,\frak{m}_2]\subset \frak{m}_3 \qquad [\frak{m}_1,\frak{m}_3]\subset \frak{m}_2 \qquad [\frak{m}_2,\frak{m}_3]\subset \frak{m}_1. \end{equation}\end{lemma}
\begin{proof}
We prove the first relation and the others follow in a similar way.
Let $X_i\in \frak{m}_i$ $(i=1,2)$.   Then
\begin{equation}\label{sin}
[X_1,X_2]=X_{\frak{k}}+X_{\frak{m}_1}+X_{\frak{m}_2}+X_{\frak{m}_3},
\end{equation}
where $X_{\frak{k}},X_{\frak{m}_i}$ are the projections of $[X_1,X_2]$ on $X_{\frak{k}},X_{\frak{m}_i}$ respectively.  By using the $B$-orthogonality of the spaces $\frak{k},\frak{m}_1,\frak{m}_2,\frak{m}_3$ and the relations $[\frak{k},\frak{m}_i]\subset \frak{m}_i$,  $[\frak{m}_i,\frak{m}_i]\subset \frak{k}$, equation (\ref{sin}) yields
\begin{equation*}B(X_{\frak{k}},X_{\frak{k}})=B([X_1,X_2],X_{\frak{k}})=B([X_{\frak{k}},X_1],X_2)=0,\end{equation*}
\begin{equation*}B(X_{\frak{m}_1},X_{\frak{m}_1})=B([X_1,X_2],X_{\frak{m}_1})=B([X_{\frak{m}_1},X_1],X_2)=0,\end{equation*} 
\begin{equation*}B(X_{\frak{m}_2},X_{\frak{m}_2})=B([X_1,X_2],X_{\frak{m}_2})=B([X_2,X_{\frak{m}_1}],X_1)=0.\end{equation*} 
Therefore, $X_{\frak{k}}=X_{\frak{m}_1}=X_{\frak{m}_2}=0$, which proves that $[\frak{m}_1,\frak{m}_2]\subset \frak{m}_3.$\end{proof}

\begin{remark}\label{lts}
If we set $\frak{g}_i=\frak{k}\oplus \frak{m}_i$ $(i=1,2,3)$,  then the $\frak{m}_i$ equipped with the operation\\
 $[[\cdot, \cdot]\cdot ,\cdot ]:\frak{m}_i\times \frak{m}_i\times \frak{m}_i\rightarrow \frak{m}_i$ is a Lie triple system of $\frak{g}$.  Therefore, $(\frak{g}_i,\frak{k})$ is a symmetric pair and the spaces $M_i=\exp(\frak{g}_i)\cdot o$ $(i=1,2,3)$, are totally geodesic symmetric submanifolds of $G/K$.
\end{remark}

\section{Geodesics in generalised Wallach spaces}

Let $M=G/K$ be a generalised Wallach space 
with a   $G$-invariant metric  determined  by  the $\operatorname{Ad}(K)$-invariant inner product on 
$\frak{m} = \frak{m}_1\oplus\frak{m}_2\oplus\frak{m}_3$ of the form
\begin{equation} \label{product} 
\langle \ , \ \rangle=\lambda_1(-B)|_{\frak{m}_1}+\lambda_2(-B)|_{\frak{m}_2}+\lambda_3(-B)|_{\frak{m}_3}
= (\lambda_1,\lambda_2,\lambda_3), \quad (\lambda _i>0).
\end{equation}

Let $X_1+X_2+X_3 \in \frak{m}$ with $X_i\in \frak{m}_i$, $(i=1,2,3)$.  We look for geodesics in $(M,g)$ through $o$ of the form 
\begin{equation} \label{geodesicform} \gamma(t)=\exp(tX) \exp(tY) \exp(tZ)\cdot o, \end{equation}
satisfying 
\begin{equation} \label{tangent} \dot{\gamma}(0)=X_1+X_2+X_3, \end{equation} and\\ 
\begin{equation} \label{linear}
\begin{array}{lll}
Z=a_1X_1+a_2X_2+a_3X_3\ \\\\

Y=b_1X_1+b_2X_2+b_3X_3\ \\\\

X=c_1X_1+c_2X_2+c_3X_3,
\end{array}
\end{equation}
where $a_i, b_i, c_i \in \mathbb R$ $(i=1,2,3)$.  
 It will turn out that if we impose the geodesic condition on such curves, then certain restrictions on the parameters $\lambda_i$ of  the $G$-invariant metric $g$ emerge. 
 
\begin{prop} \label{restriction} 
Let $M=G/K$ be a generalised Wallach space equipped with a $G$-invariant metric $g$ determined by the scalar product $\langle \ ,\  \rangle=(\lambda_1,\lambda_2,\lambda_3)$.  Let $\gamma:I\subset \mathbb R\rightarrow M$ be a geodesic on $(M,g)$ of the form $\gamma(t)=\exp(tX)\exp(tY)\exp(tZ)\cdot o$, where $X,Y,Z \in \frak{m}$ are given by (\ref{linear}), such that $\gamma(0)=o$, $\dot{\gamma}(0)=X_1+X_2+X_3 \in \frak{m}$ ($X_i\in \frak{m}_i$, $(i=1,2,3)$) and $[X_i, X_j]\neq 0$ for $i\neq j$.  Then (up to scalar) one of the following possibilities is valid:\\

\textbf{1})\qquad $\langle \ ,\  \rangle=(1,1,c)$ and $\gamma(t)=\exp t(X_1+X_2+cX_3)\exp t(1-c)X_3\cdot o$\\

\textbf{2})\qquad $\langle \ , \ \rangle=(1,c,1)$  and $\gamma(t)=\exp t(X_1+cX_2+X_3)\exp t(1-c)X_2\cdot o$\\

\textbf{3})\qquad $\langle \ , \ \rangle=(c,1,1)$  and $\gamma(t)=\exp t(cX_1+X_2+X_3)\exp t(1-c)X_1\cdot o$,\\

where $c>0$.
\end{prop}
 
\begin{proof} Since the
diagonal metric $(\lambda_1,\lambda_2,\lambda_3)$  is a scalar multiple of the metric $(1,\frac{\lambda_2}{\lambda_1},\frac{\lambda_3}{\lambda_1})$,   Koszul's formula implies that  if $\gamma$ is an unparametrised geodesic on a Riemannian manifold $(M,g)$, then $\gamma$ is also an unparametrised geodesic on $(M,cg)$ $(c\in \mathbb R^+)$.  Therefore, for our purposes we need only to consider metrics of the form 
$(1,\lambda_2,\lambda_3)$.
Relation (\ref{geodesicform}) implies that $\dot{\gamma}(0)=X+Y+Z$ so
condition (\ref{tangent}) and relations (\ref{linear}) imply that 
\begin{equation} \label{sum} 
a_i+b_i+c_i=1, \qquad  (i=1,2,3).
\end{equation} 
Assume that the curve (\ref{geodesicform}) is a geodesic.  For any $W\in \frak{m}$, Proposition \ref{nikos} implies that 
\begin{eqnarray} \label{system1} 
&&G_W(0)=0\\
\label{system2} &&\dot{G}_W(0)=0.
\end{eqnarray} 
By solving the above system, we will determine the values of the coefficients $a_i, b_i, c_i$ and $\lambda_i$.
   By using (\ref{*}), (\ref{T}) and equation (\ref{system1}) it follows that 
\begin{equation} \label{system11}
\begin{array}{lll}
 G_W(0)&=&g(X_{\frak{m}}+Y_{\frak{m}}+Z_{\frak{m}},[W, X+Y+Z]_{\frak{m}})_o\ \\\\

&+&g(W,[X, Y+Z]_{\frak{m}}+[Y, Z]_{\frak{m}})_o=0.
\end{array}
\end{equation}
Since equation (\ref{system11}) is true for all $W$ and in view of decomposition (\ref{dec}), we need to consider the following three simultaneous cases:

$$ a)\ W\in \frak{m}_1, \qquad 
b)\ W\in \frak{m}_2,  \qquad
c)\  W\in \frak{m}_3.
$$
For case $a)$ we use that $\dot{\gamma}(0)=X+Y+Z=X_1+X_2+X_3$, so equation (\ref{system11}) is equivalent to
\begin{equation} \label{system111}
\begin{array}{lll}
&& g([W, X_1+X_2+X_3]_{\frak{m}}, X_1+X_2+X_3)_o\ \\\\

&&\ \ \ +g(W, [X,X_1+X_2+X_3]_{\frak{m}}+[Y, Z]_{\frak{m}})_o=0. 
\end{array}
\end{equation}
By using expressions (\ref{linear}) for $X,Y,Z$, equation (\ref{sum}), the bracket relations (\ref{brackets}) and the orthogonality of the spaces $\frak{k},\frak{m}_1,\frak{m}_2,\frak{m}_3$,  equation (\ref{system111}) gives
\begin{equation} \label{system1111}
\begin{array}{lll}
 && g([W, X_2], X_3)_o+g([W, X_3], X_2)_o+(a_3-a_2+b_3-b_2)g(W, [X_2, X_3])_o\ \\\\

&&\ \ +(b_2a_3-b_3a_2)g(W, [X_2, X_3])_o=0.
\end{array}
\end{equation}
Also, by using (\ref{product}) and the $\operatorname{ad}$-skew symmetry of the Killing form of $\frak{g}$, equation (\ref{system1111}) implies that
\begin{equation*}
\Big(a_3-a_2+b_3-b_2+b_2a_3-b_3a_2-\lambda_2+\lambda_3\Big)B(W, [X_2, X_3])=0,
\end{equation*}
for every $W\in \frak{m}_1$, which gives that
\begin{equation} \label{EQ1} a_3-a_2+b_3-b_2+b_2a_3-b_3a_2=\lambda_2-\lambda_3. \end{equation}
Similarly, for cases $b)$ and $c)$, equation (\ref{system11}) yields 
\begin{eqnarray} 
\label{EQ2} a_3-a_1+b_3-b_1+b_1a_3-b_3a_1=\frac{1-\lambda_3}{\lambda_2}\\
\label{EQ3} a_2-a_1+b_2-b_1+b_1a_2-b_2a_1=\frac{1-\lambda_2}{\lambda_3},
\end{eqnarray} respectively.  
Therefore, condition $G_W(0)=0$ is equivalent to the system of equations (\ref{EQ1})-(\ref{EQ3}).\\
To compute $\dot{G}_W(0)$ we use relation (\ref{lab1}) to differentiate (\ref{*}).
   By using the Jacobi identity and $\dot{\gamma}(0)=X+Y+Z=X_1+X_2+X_3$ it follows that
\begin{equation} \label{system22}
\begin{array}{lll}
&& \dot{G_W}(0)=g([X+Y,Y+Z]_{\frak{m}}, [W, X_1+X_2+X_3]_{\frak{m}})_o\ \\\\

&&\ +g(W, [[X_1+X_2+X_3,Z],Z]_{\frak{m}})_o+g(W, [[X,Y], Y+2Z]_{\frak{m}})_o\ \\\\

&&\ +g(X_1+X_2+X_3, [W, [Y,Z]]_{\frak{m}})_o\ \\\\

&&\ -g([[X,X_1+X_2+X_3],W]_{\frak{m}}, X_1+X_2+X_3)_o=0.
\end{array}
\end{equation}    
As before, by the linearity of (\ref{system22})  we need to consider the following simultaneous cases:
$$ a)\ W\in \frak{m}_1, \qquad 
b)\ W\in \frak{m}_2,  \qquad
c)\  W\in \frak{m}_3.
$$
For case $a)$ we use the orthogonality of the spaces $\frak{k},\frak{m}_1,\frak{m}_2,\frak{m}_3$, the inner product (\ref{product}) and relations (\ref{linear}) and (\ref{sum}), and equation (\ref{system22}) reduces to

\begin{eqnarray*}
\begin{array}{lll}
\Big((1-\lambda_2)b_2+2(1-\lambda_2)a_2+\lambda_3b_2a_1-\lambda_3b_2a_2-\lambda_3a_2^2+\lambda_3a_1a_2\ \\\\
-(\lambda_3-\lambda_2)(1-\lambda_2)\Big)B([X_1, X_2], [W, X_2]) \\\\
+\Big((1-\lambda_3)b_3+2(1-\lambda_3)a_3+\lambda_2b_3a_1-\lambda_2b_3a_3-\lambda_2a_3^2+\lambda_2a_1a_3\ \\\\
-(\lambda_2-\lambda_3)(1-\lambda_3)\Big)B([X_1, X_3], [W, X_3])=0,
\end{array}
\end{eqnarray*}

from which we obtain the following two equations:
\begin{equation} 
\label{EQ4}
\begin{array}{lll}
&& \noindent (1-\lambda_2)b_2+2(1-\lambda_2)a_2+\lambda_3b_2a_1-\lambda_3b_2a_2-\lambda_3a_2^2
   +\lambda_3a_1a_2\ \\\\

&&\ =(\lambda_3-\lambda_2)(1-\lambda_2),
\end{array}
\end{equation}

\begin{equation} \label{EQ5}
\begin{array}{lll}
&& \noindent (1-\lambda_3)b_3+2(1-\lambda_3)a_3+\lambda_2b_3a_1-\lambda_2b_3a_3-\lambda_2a_3^2+\lambda_2a_1a_3\ \\\\

&&\ =(\lambda_2-\lambda_3)(1-\lambda_3).
\end{array}
\end{equation}

Similarly, cases $b)$ and $c)$ yield the equations:
\begin{equation}
\begin{array}{lll} 
 \label{EQ6} &&\lambda_2(1-\lambda_2)b_1+2\lambda_2(1-\lambda_2)a_1+\lambda_2\lambda_3b_1a_1-\lambda_2\lambda_3b_1a_2+\lambda_2\lambda_3a_1^2-\lambda_2\lambda_3a_1a_2\ \\\\

&&\ \ =(\lambda_2-1)(1-\lambda_3),
\end{array}
\end{equation}

 \begin{equation}
\begin{array}{lll}
 \label{EQ7} && \lambda_2(\lambda_2-\lambda_3)b_3+2\lambda_2(\lambda_2-\lambda_3)a_3+\lambda_2b_3a_2-\lambda_2b_3a_3-\lambda_2a_3^2+\lambda_2a_2a_3\ \\\\

&&\ \ =(1-\lambda_3)(\lambda_2-\lambda_3),
\end{array}
\end{equation}

 \begin{equation}
\begin{array}{lll}
 \label{EQ8} && \lambda_3(1-\lambda_3)b_1+2\lambda_3(1-\lambda_3)a_1+\lambda_2\lambda_3b_1a_1-\lambda_2\lambda_3b_1a_3+\lambda_2\lambda_3a_1^2-\lambda_2\lambda_3a_1a_3\ \\\\

&&\ \ =(\lambda_2-1)(1-\lambda_3),
\end{array}
\end{equation}

\begin{equation}
\begin{array}{lll}
 \label{EQ9} && \lambda_3(\lambda_2-\lambda_3)b_2+2\lambda_3(\lambda_2-\lambda_3)a_2+\lambda_3b_2a_2-\lambda_3b_2a_3+\lambda_3a_2^2-\lambda_3a_2a_3\ \\\\

&&\ \ =(\lambda_2-\lambda_3)(1-\lambda_2).
\end{array}
\end{equation}

Therefore, the equation $\dot{G}_W(0)=0$ is equivalent to the system of equations (\ref{EQ4})-(\ref{EQ9}).\\
  To summarise, the system of equations (\ref{system1})-(\ref{system2}) is equivalent to the system (\ref{EQ1})-(\ref{EQ3}), (\ref{EQ4})-(\ref{EQ9}).  By using a program of   symbolic computation we  obtain the following relations among the variables $a_i, b_i$, $(i=1,2,3)$, $\lambda_1,\lambda_2$:

\begin{equation} \label{s1} \lambda_2=1, a_1=a_2=b_1=b_2=0, b_3=1-a_3-\lambda_3\end{equation}

\begin{equation} \label{s2} \lambda_2=1, a_1=a_2=0, a_3=1-\lambda_3, b_1=b_2, b_3=\lambda_3b_2\end{equation}

\begin{equation} \label{s3} \lambda_3=1, a_1=a_3=b_1=b_3=0, b_2=1-a_2-\lambda_2\end{equation}

\begin{equation} \label{s4} \lambda_3=1, a_1=a_3=0, a_2=1-\lambda_2, b_1=b_3=\frac{b_2}{\lambda_2}\end{equation}

\begin{equation} \label{s5} \lambda_2=\lambda_3, a_2=a_3=0, a_1=\frac{\lambda_3-1}{\lambda_3}, b_2=b_3=\lambda_3b_1\end{equation}

\begin{equation} \label{s6} \lambda_2=\lambda_3, a_2=a_3=b_2=b_3=0, b_1=\frac{\lambda_3-a_1\lambda_3-1}{\lambda_3}.\end{equation}

By (\ref{s1}) and setting $c=\lambda_3$ we obtain that $\langle \ , \ \rangle=(1,1,c)$.  Using relations (\ref{linear}), we have that $X=X_1+X_2+cX_3, Y=(1-a_3-c)X_3, Z=a_3X_3$ and since $Y$ is parallel to $Z$, (\ref{geodesicform}) implies that
$$
\gamma(t)=\exp(tX)\exp t(Y+Z)\cdot o=\exp t(X_1+X_2+cX_3)\exp t(1-c)X_3\cdot o,
$$
which proves conclusion \textit{\textbf{1)}} of the proposition.
Solution (\ref{s2}) also yields $g=(1,1,c)$ and $X=(1-b_1)X_1+(1-b_1)X_2+(1-b_1)cX_3,Y=b_1X_1+b_1X_2+cb_1X_3,Z=(1-c)X_3$.  Using (\ref{geodesicform}) and since $X$ is parallel to $Y$ we have that $\gamma(t)=\exp t(X+Y) \exp (tZ)\cdot o=\exp t(X_1+X_2+cX_3)\exp t(1-c)X_3\cdot o.$  This also yields conclusion \textit{\textbf{1)}} of the proposition.
By setting $c=\lambda_2$ similar computations imply that solutions (\ref{s3}) and (\ref{s4}) give
 conclusion \textit{\textbf{2)}} of the proposition. Finally, for $c=\frac{1}{\lambda_2}$  solutions (\ref{s5}) and (\ref{s6}) imply conclusion
\textit{\textbf{3)}}. 
 \end{proof} 

Next we show that the three curves obtained in Proposition \ref{restriction} are indeed geodesics.
\begin{theorem} \label{final} Let $G$ be a connected Lie group and $M=G/K$ be a generalised Wallach space.  If the $G$-invariant metric  (\ref{metric}) has one of the forms 
 $(1,1,c)$, $(1,c,1)$ or $(c,1,1)$ (up to scalar), then the unique geodesic $\gamma (t)$ passing through $o$ with 
$\dot{\gamma}(0)=X_1+X_2+X_3 \in T_o(G/K)$, $X_i\in \frak{m}$, is given by

\begin{eqnarray}\label{main}
\gamma(t)&=&\exp t(X_1+X_2+cX_3)\exp (t(1-c)X_3)\cdot o,\\
\gamma(t)&=&\exp t(X_1+cX_2+X_3)\exp (t(1-c)X_2)\cdot o,\\
\gamma(t)&=&\exp t(cX_1+X_2+X_3)\exp (t(1-c)X_1)\cdot o
\end{eqnarray}

respectivelly. 
\end{theorem}
\begin{proof} We will prove equation (\ref{main}) and the others can be shown by a similar manner.  Assume that the $G$-invariant metric is given by $\langle \ , \ \rangle=(1,1,c)$.  
The curve (\ref{main}) is of the form (\ref{geodesicform}) in which, without loss of generality, we set $X=X_1+X_2+cX_3, Y=(1-c)X_3, Z=0$.  Let $W\in \frak{m}$ be arbitrary.  We need to verify proposition \ref{nikos} for $G_W(t)$ given by (\ref{*}) (and with $Z=0$).  This is equivalent to 
\begin{equation} \label{**}G_W(t)=g(W, [TX, TY]_{\frak{m}})_o+g([W,TX+TY]_{\frak{m}}, TX+TY)_o=0.
\end{equation}
 Since $TX+TY=T(X+Y)=T(X_1+X_2+X_3)=TX_1+TX_2+TX_3$, we obtain that
\begin{equation} \label{***}
\begin{array}{lll}
&& G_W(t)=(1-c)g(W, [TX_1, TX_3]_{\frak{m}})_o+(1-c)g(W, [TX_2, TX_3]_{\frak{m}})_o\ \\\\

&&\ +g([W, TX_1]_{\frak{m}}, TX_1)_o+g([W,TX_1]_{\frak{m}},TX_2)_o+g([W,TX_1]_{\frak{m}}, TX_3)_o\ \\\\

&&\ +g([W, TX_2]_{\frak{m}}, TX_1)_o+g([W, TX_2]_{\frak{m}}, TX_2)_o+g([W, TX_2]_{\frak{m}}, TX_3)_o\ \\\\

&&\ +g([W, TX_3]_{\frak{m}}, TX_1)_o+g([W, TX_3]_{\frak{m}}, TX_2)_o+g([W, TX_3]_{\frak{m}}, TX_3)_o.     
\end{array}
\end{equation}

Since $Z=0$, relation (\ref{T}) implies that $T=\operatorname{Ad}(\exp(-tY))$ and by equation (\ref{lab2}) we obtain that $TY=Y$ and $TX_3=X_3$.  Moreover, since $G$ is connected, then the definition of $T$ implies that $TX_i\in [\frak{m}_3,\frak{m}_i]$ for $i=1,2$.  Therefore, $[TX_1,TX_3]\in \frak{m}_1$ and $[TX_2,TX_3]\in \frak{m}_2$.  Also, since $\langle  \ , \ \rangle=(-B)|_{\frak{m}_1}+(-B)|_{\frak{m}_2}+c(-B)|_{\frak{m}_3}$, the $\operatorname{ad}$-skew symmetry of the Killing form of $\frak{g}$ implies that $g([W, TX_i]_{\frak{m}}, TX_i)_o=B([W, TX_i]_{\frak{m}}, TX_i)=0$.  From the above discussion and the orthogonality of spaces $\frak{k},\frak{m}_1,\frak{m}_2,\frak{m}_3$, relation (\ref{***}) reduces to

\begin{equation*}
\begin{array}{lll}
&& G_W(t)=(1-c)B(W, [TX_1, TX_3])+(1-c)B(W, [TX_2, TX_3])\ \\\\

&&\ \ \ +B([W, TX_1], TX_2)+cB([W,TX_1],TX_3)+B([W,TX_2], TX_1)\ \\\\

&&\ \ \ +cB([W, TX_2], TX_3)+B([W, TX_3], TX_1)+B([W, TX_3], TX_2)_o\ \\\\

&&\ =B(W,[TX_1,TX_3])(1-c+c-1)+B(W,[TX_2,TX_3])(1-c+c-1)\ \\\\

&&\ \ \ +B(W,[TX_1,TX_2])(1-1)=0,     
\end{array}
\end{equation*}  
and this proves the theorem.
\end{proof}

The following proposition is an interesting consequence of Proposition \ref{nikos}.
\begin{prop}\label{remark} Let $G/K$ be a generalised Wallach space with $\fr{m}=\fr{m}_1\oplus \fr{m}_2\oplus \fr{m}_3$.  If any of the relations 
\begin{equation} 
\label{q}[\fr{m}_i,\fr{m}_j]=0\quad ( i\neq j )
\end{equation}
hold, then $G/K$ is a g.o. space for any diagonal metric $(1,\lambda_2, \lambda_3)$. 
That is any geodesic $\gamma$ of $G/K$ passing through $o$ and tangent to any $X\in \fr{m}$ is given by 

\begin{equation*}\gamma(t)=\exp(tX)\cdot o. \end{equation*}
\end{prop}

\begin{proof} Assume that $[\fr{m}_1,\fr{m}_2]=0$ (the other cases can be treated similarly).  We set $Y=Z=0$ in (\ref{T}) and in (\ref{*}).  Then $TX=X$ and $G_W(t)=g(X_{\fr{m}},[W,X]_{\fr{m}})$.\\
We choose $X=X_1+X_2+X_3\in \fr{m}$ and let $W=W_1+W_2+W_3\in \fr{m}$ be arbitrary $(X_i,W_i\in \fr{m}_i$, $i=1,2,3)$.  By using the assumption, the $\operatorname{ad}$-skew symmetry of the Killing form $B$, and relations (\ref{brackets}) and (\ref{rel3}) it follows that  
\begin{equation*}
\begin{array}{lll}
&&G_W(t)=g(X_{\fr{m}},[W,X]_{\fr{m}})=B(X_1,[W_2,X_3])+B(X_1,[W_3,X_2])\ \\\\

&&\ \ \qquad+\lambda_2B(X_2,[W_1,X_3])+\lambda_2B(X_2,[W_3,X_1])\ \\\\

&&\ \ \qquad+\lambda_3B(X_3,[W_1,X_2])+\lambda_3B(X_3,[W_3,X_1])=0.
\end{array}
\end{equation*} 

Therefore, Proposition \ref{nikos} implies that $\gamma(t)=\exp t(X_1+X_2+X_3)\cdot o$, which is a homogeneous geodesic.
\end{proof}

Notice that the condition $G_W(t)=g(X_{\fr{m}},[W,X]_{\fr{m}})=0$ is the well known condition for homogeneous geodesics originally presented in \cite{KV}.
\begin{example}  
\textnormal{
Let $M=E_6/(U(1)\times U(1)\times \operatorname{Spin}(8))$. 
This is  a generalized flag manifold whose isotropy representation decomposes into three irreducible non equivalent submodules.  These spaces have been classified in \cite{Ki}.
  Its painted Dynking diagram is obtained from the 
  Dynkin diagram of $E_6$, by painting black the simple roots $a_1,a_5$ with Dynkin mark $1$, as shown below.
}
\smallskip
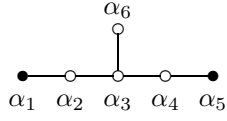
\begin{center}
\begin{picture}(160,35)(-25,3)

\put(15, 10){\circle*{4}}
\put(15,0){\makebox(0,0){$\al_1$}}
\put(17, 10){\line(1,0){14}}
\put(33, 10){\circle{4.4}}
\put(33,0){\makebox(0,0){$\al_2$}}
\put(35, 10){\line(1,0){14}}
\put(51, 12){\line(0,1){14}}
\put(51, 10){\circle{4}}
\put(51,0){\makebox(0,0){$\al_3$}}
\put(51, 28){\circle{4}}
\put(51,35){\makebox(0,0){$\al_6$}}
\put(53,10){\line(1,0){14}}
\put(69,10){\circle{4}}
\put(69,0){\makebox(0,0){$\al_4$}}
\put(71,10){\line(1,0){14}}
\put(87,10){\circle*{4}}
\put(87,0){\makebox(0,0){$\al_5$}}

\end{picture}
\end{center}

\medskip
\noindent
\textnormal{
Let $R$ be the root system of $E_6$.  For $\al\in R$ we denote by $\frak{g}^\al$ the corresponding root space.  Any $\al\in R$ can be expressed as $\al=\sum k_i\al _i$ where $\al _i$ are the simple roots of $E_6$. 
We put $\Delta_{(1,0)}=\left\{\al\in R: k_1=1, k_5=0\right\}$, $\Delta_{(0,1)}=\left\{\al\in R:k_1=0, k_5=1\right\}$ and
$\Delta_{(1,1)}=\left\{\al\in R: k_1=1, k_5=1\right\}$.  
 Then the spaces
}
$$
\frak{m}_1=\sum_{\al\in \Delta_{(1,0)}}\frak{g}^\al, 
\quad \frak{m}_2=\sum_{\al\in \Delta_{(0,1)}}\frak{g}^\al, \quad
\frak{m}_3=\sum_{\al\in \Delta_{(1,1)}}\frak{g}^\al
$$
\noindent
\textnormal{satisfy the bracket relations (\ref{brackets}) and $\frak{m}=\frak{m}_1\oplus \frak{m}_2 \oplus \frak{m}_3$.
Moreover $g=(1,1,2)$ is a K\"ahler-Einstein metric of the form $(1,1,c)$.  Therefore $(M,g)$ satisfies the conditions of Theorem \ref{final} and the unique geodesic $\gamma$ passing through $o$ with $\dot{\gamma}(0)=X_1+X_2+X_3$, $X_i\in \frak{m}_i$ is given by
$$
\gamma(t)=\exp t(X_1+X_2+2X_3) \exp t(-X_3)\cdot o.
$$
According to \cite{Ki} the invariant  metrics $g^{1}=(1,2,1)$ and $g^{2}=(2,1,1)$ on 
$$
M=E_6/(U(1)\times U(1)\times \operatorname{Spin}(8))
$$ 
are also K\"ahler-Einstein, so the spaces $(M,g^{1})$, $(M,g^{2})$ also satisfy the conditions of Theorem \ref{maintheorem}.}
\end{example}

\section{Riemannian submersions}
A natural question raised in this paper is why the search for geodesics of the form (\ref{3exp}) in a generalized Wallach space, leads to the geodesics obtained in Proposition \ref{restriction} (with two exponential terms instead of three).
One way to explain this is to use the fact that to each generalized Wallach space $G/K$ there is an associated fibration $G/K\to G/H$, where the base space $G/H$ and the fiber $H/K$ are locally symmetric spaces (cf. \cite[pp. 46, Remark 1]{NRS}).  More precisely, let 
$\frak{g}=\frak{k}\oplus\frak{m}=
\frak{k}\oplus\frak{m}_1\oplus\frak{m}_2\oplus\frak{m}_3
$
be a reductive decomposition of $\frak{g}$ with $[\frak{m}_i, \frak{m}_i]\subset\frak{k}$, and
let $\frak{g}_i=\frak{k}\oplus\frak{m}_i$. If $G_i$ is the simply connected Lie group with Lie algebra
$\frak{g}_i$,  then it can be shown that the spaces $G_i/K$ and $G/G_i$ are locally symmetric and there exist
   fibrations
\begin{equation}\label{fibrations}
G_i/K\to G/K\to G/G_i,\quad (i=1, 2, 3).
\end{equation}
The corresponding reductive decompositions of $G_i/K$ and $G/G_i$ are  $\frak{g}_i=\frak{k}\oplus\frak{m}_i$ and
$\frak{g}=(\frak{k}\oplus\frak{m}_i)\oplus (\frak{m}_k\oplus\frak{m}_l)$ ($i, k, l$ distinct).
If the spaces $G_i/K, G/K$ and $G/G_i$ are equipped with an invariant metric then we obtain  Riemannian submersions with totally geodesic fibers.

We consider the $G$-invariant metric on $G/K$ given by
$$
g^{\mbox{sub}}=\lambda\left.(-B)\right|_{\frak{m}_k\oplus\frak{m}_l}+\mu\left.(-B)\right|_{\frak{m}_i},\quad
(\lambda, \mu >0),
$$
which is a special case of the $G$-invariant metric (\ref{metric}).
Since the base and  the fiber are  symmetric spaces, for which is well known that geodesics are given by one-parameter subgroups, it is natural to search for geodesics of the form
$\gamma(t)=\exp (tX)\exp(tY)\cdot o$, where $X\in\frak{m}_k\oplus\frak{m}_l$, $Y\in\frak{m}_i$ with respect to the submersion metric $g^{\mbox{sub}}$.
According to Theorem \ref{final} such geodesics are obtained if we take $\lambda =c$ and $\mu =1$.
Furthermore,  Proposition \ref{restriction} sheds light to the question whether there exists a $G$-invariant metric of the form $(1, \lambda _2, \lambda _3)$ so that the geodesics in a generalized Wallach space $G/K$ with the initial conditions stated there, are of the form (\ref{3exp}).

\begin{example}
\textnormal{
Let $G/K=SO(l+m+n)/(SO(l)\times SO(m)\times SO(n))$ with $n\ge m\ge l\ge 1$ and $(l, m)\ne (1, 1)$.  If we set
$G_i=SO(l+m)\times SO(n)$ then fibration (\ref{fibrations}) is given by
$$
SO(l+m)/(SO(l)\times SO(m))\to G/K\to SO(l+m+n)/(SO(l+m)\times SO(n)).
$$
The special case $l=m=1$ corresponds to the Stiefel manifold $SO(n+2)/SO(n)$.  Even though the isotropy representation of $SO(n+2)/SO(2)$ contains two equivalent summands, it has been shown in
\cite{Ke} that any $SO(n+2)$-invariant metric is diagonal.
}
\end{example}
\begin{example} 
\textnormal{Let $G/K=E_6/(U(1)\times U(1)\times \operatorname{Spin}(8))$ be the generalized flag manifold with $\fr{m}=\fr{m}_1\oplus \fr{m}_2\oplus \fr{m}_3$.
Let $\frak{g}_1=\frak{k}\oplus \frak{m}_1$ and let $G_1$ be the corresponding simply connected Lie group.
  It is easy to see that with the above notation $G/G_1$ is a symmetric space.  Indeed, let $\frak{m}'=T_o(G/G_1)=\frak{m}_2\oplus \frak{m}_3$.  Then it is 
$[\frak{m}', \frak{m}']\subset\frak{g}_1$ and  $[\frak{g}_1, \frak{m}']\subset \frak{m}'$.
To find the precise quotient  $G/G_1$, we look at all possible irreducible compact symmetric spaces of the form $E_6/G_1$, and these are
$E_6/Sp(4),  E_6/(SU(6)\times SU(2)), E_6/(SO(10)\times U(1))$ and $E_6/F_4$.
Our choice will be determined by the dimension of $\frak{g}_1$.
It is $\dim_{\mathbb R}\frak{k}=30$, and
  according to \cite[Proposition 2.6, p. 311]{Ki} it is
$\dim_{\mathbb R}\frak{m}_1=\dim_{\mathbb R}\frak{m}_2=\dim_{\mathbb R}\frak{m}_3=16$. 
 Hence, $\dim_{\mathbb R}\frak{g}_1=46=\dim_{\mathbb R}(\frak{so}(10)\oplus\frak{u}(1))$ so
  we conclude that $G/G_1=E_6/(SO(10)\times U(1))$.  Therefore,  fibration (\ref{fibrations}) is given by
\begin{equation*}
SO(10)/(SO(2)\times SO(8))\rightarrow E_6/(U(1)\times U(1)\times \operatorname{Spin}(8)) \rightarrow E_6/(SO(10)\times U(1)).
\end{equation*}  
}
\end{example}

\end{document}